\documentclass[10pt]{amsart}
\usepackage{lipsum}
\makeatletter
\g@addto@macro{\endabstract}{\@setabstract}
\newcommand{\authorfootnotes}{\renewcommand\thefootnote{\@fnsymbol\c@footnote}}%
\makeatother
\usepackage{amssymb}
\usepackage[numbers,sort&compress]{natbib}
\renewcommand{\leq}{\leqslant}
\renewcommand{\geq}{\geqslant}
 \makeatletter \numberwithin{equation}{section}
\numberwithin{figure}{section} 
\theoremstyle{plain}
\newtheorem{thm}{Theorem}[section]

\newtheorem{lem}[thm]{Lemma} 
\theoremstyle{Definition}

\begin{document}

\begin{center}
  \LARGE
   Choldowsky type generalization of
the $q-$Favard-Sz\`{a}sz operators \par \bigskip

  \normalsize
  \authorfootnotes
  Ali Karaisa \footnote{Corresponding author: Tel:+90 332 323 8220; fax:+90
332 323 8245 E-mail:akaraisa@konya.edu.tr}  \par \bigskip

Department of MathematÝcs-Computer Sciences, Faculty of Science,
Necmettin Erbakan Unýversity Meram Campus , 42090 Meram, Konya,
Turkey \par
  \par \bigskip

\end{center}

\begin{abstract}
In the present paper, we introduce a  Choldowsky type generalization
of the $q-$Favard-Sz\`{a}sz operators and obtain weighted
statistical approximation properties of these operators. We also
establish the rates of statistical convergence by means of the
modulus of continuity and the Lipschitz type maximal function.
Further, we study the local approximation properties of these
operators.


\end{abstract}

\section{Introduction}
The approximation of function by using the linear positive operators
introduced via $q-$ calculus is currently be come active research.
After the paper of Phillips \cite{phillips} who generalized the
classical Bernstein polynomials base on $q-$integer, many
generalization of well-known positive linear operators, base on
$q-$integer were introduce and studied several authors. We refer to
the following studies related to the approximation properties of the
$q-$analogues of some operators, \cite{q1,q2,1,2}. Recently, the
statistical approximation properties have been invesigated for
$q-$analogue polynomials. For example in \cite{11,12} $q-$analogue
of Kantorovich-Bernstein operators;  in \cite{13}
$q-$Baskakov-Kantorovich operators; \cite{14} $q-$Sz\`{a}sz-Mirakjan
operators; in \cite{mursaleen} modified $q-$Stancu-Beta operators
were introduced and their statistical properties were investigated.
For further information related to the statistical approximation of
the operators, the followings are remarkable among others
\cite{5,6,m,karaisa}.

In this work, we generalized a  Choldowsky type Favard-Sz\`{a}sz
operators base on $q-$integer and we study the weighted statistical
approximation properties of the Choldowsky type $q-$Favard-Sz\`{a}sz
operators via Korovkin type approximation theorem. Further we
compute  the rate of statistical convergence by using modulus of
continuity. Furthermore we also obtain some local approximation
results of these new operators.

First of all, we recall some definitions and notations regarding the
concept of $q-$calculus. For any none-negative integer $r$, the $q-$
integer of the number $r$ is defied by
\begin{eqnarray*}
\left[ r\right] _{q}:=\left\{
\begin{tabular}{ll}
$\frac{1-q^{r}}{1-q},$ & if $q\neq1$ \\
$1,$ & if $q=1$%
\end{tabular}%
\right..
\end{eqnarray*}

The $q-$factorial $\left[ n\right] _{q}!$ and $q-$binomial
coefficients  are defined as
\begin{eqnarray*}
\left[ n\right] _{q}!:=\left\{
\begin{tabular}{ll}
$\left[ n\right] _{q}\left[ n-1\right] _{q}\cdots \left[ 1\right] _{q},$ & $%
n\in\mathbb{N}$ \\
$1,$ & $n=0$%
\end{tabular}%
\right.
\end{eqnarray*}
and
\begin{eqnarray*}
\left[
\begin{array}{c}
n \\
k
\end{array}
\right] _{q}=\frac{\left[ n\right] _{q}!}{\left[ k\right]
_{q}!\left[ n-k \right] _{q}!},0\leq k\leq n.
\end{eqnarray*}
The $q-$derivative $D_{q}f$ of a function $f$ is defined by
\begin{eqnarray*}
\left( D_{q}f\right) \left( x\right) =\frac{f\left( x\right)
-f\left( qx\right) }{\left( 1-q\right) x}, \ x\neq 0.
\end{eqnarray*}
Also, if there exists $\frac{df}{dx}\left( 0\right) $, then $ \left(
D_{q}f\right) \left( 0\right) =\frac{df}{dx}\left( 0\right) $. The
following $q-$derivatives of the product of the functions $f(x)$ and
$g(x)$ are equivalent:
\begin{eqnarray*}
D_{q}\left(f(x)g(x)\right) =f\left( qx\right)
D_{q}g\left(x\right)+g(x)D_{q}f\left(x\right)
\end{eqnarray*}
and
\begin{eqnarray*}
D_{q}\left(f(x)g(x)\right) =f\left( x\right)
D_{q}g\left(x\right)+g(q x)D_{q}f\left(x\right).
\end{eqnarray*}
The
$q-$analogues of the exponential function are given by%
\begin{eqnarray*}
e_{q}^{x}=\sum\limits_{n=0}^{\infty }\frac{x^{n}}{\left[ n\right]
_{q}!}
\end{eqnarray*}
and
\begin{eqnarray*}
E_{q}^{x}=\sum\limits_{n=0}^{\infty }q^{\frac{n\left( n-1\right) }{2}}\frac{%
x^{n}}{\left[ n\right] _{q}!}.
\end{eqnarray*}
The exponential functions have the following properties:%
\begin{eqnarray*}
D_{q}\left( e_{q}^{ax}\right) =ae_{q}^{ax},\ D_{q}\left(
E_{q}^{ax}\right) =aE_{q}^{aqx},\
e_{q}^{x}E_{q}^{-x}=E_{q}^{x}e_{q}^{-x}=1.
\end{eqnarray*}
Further results can be found in \cite{kitap 1}.

\section{Construction of the operators}

In \cite{8}, Jakimovski and Leviatan\ introduced a Favard Sz\`{a}sz
type operator, by using Appell polynomials $p_{k}\left( x\right)\geq
0$ defined by%
\begin{equation*}
g\left( u\right) e^{-ux}=\sum\limits_{k=0}^{\infty }p_{k}\left(
x\right) u^{k},
\end{equation*}
where $g\left( z\right) =\sum\limits_{n=0}^{\infty }a_{n}z^{n}$ is
an analytic function in the disc $\left\vert z\right\vert <R$, $R>1$ and $%
g\left( 1\right) \neq 0$,
\begin{equation*}
P_{n,t}\left( f;x\right) =\frac{e^{-nx}}{g\left( 1\right) }%
\sum\limits_{k=0}^{\infty }p_{k}\left( nx\right) f\left(\frac{k}{n}%
\right)
\end{equation*}
and they investigated some approximation properties of these
operators.

Atakut at all. \cite{4} defined  a Choldowsky type of
Favard-Sz\`{a}sz operators as follows:
\begin{equation*}
P^{\ast}(f;x)=\frac{e^{-\frac{nx}{b_{n}}}}{g\left( 1\right) }%
\sum\limits_{k=0}^{\infty }p_{k}\left( \frac{nx}{b_{n}}\right)
f\left( \frac{k }{n }b_{n}\right)
\end{equation*}
with $b_{n}$ a positive increasing sequence with the properties
$\lim_{n \to \infty}b_{n}=\infty$ and $\lim_{n \to
\infty}\frac{b_{n}}{n}=0$. They also studied some approximation
properties of the operators.

Now, let us define Choldowsky type generalization of the
$q-$Favard-Sz\`{a}sz operators as follows:
\begin{equation}
P^{\ast}_{n}\left( f;q;x\right)
=\frac{E_{q}^{-\frac{[n]_{q}}{b_{n}}x}}{A\left(1\right)}\sum_{k=0}^{\infty}\frac{P_{k}\left(
q;\frac{[n]_{q}}{b_{n}}x\right) }{\left[ k\right] _{q}!}f\left(
\frac{\left[ k\right] _{q}}{\left[ n\right] _{q} }b_{n}\right),
\label{A}
\end{equation}
where $\{ P_{k}( q;.)\}\geq 0$ is a $q-$Appell polynomial set which
is generated by
\begin{equation}\label{1}
A\left( u\right) e_{q}^{\frac{[n]_{q}}{b_{n}}xu}=\sum_{k=0}^{\infty }\frac{%
P_{k}\left( q;\frac{[n]_{q}}{b_{n}}x\right) u^{k}}{\left[ k\right]
_{q}!}
\end{equation} and $A(u)$ is defined by
\begin{eqnarray*}
A(u)=\sum_{k=0}^{\infty }a_{k}u^{k}.
\end{eqnarray*}

\section{weighted statistical approximation properties}

In this section, we give Korovkin type weighted  statistical
approximation properties of the operators $P^{\ast}_{n}\left(
f;q;x\right)$. Before proceeding further, let us give basic
definition and notation on the concept of the statistical
convergence which was introduced by Fast \cite{9}. Let $K$ be a
subset of $\mathbb{N}$, the set of natural numbers. Then,
$K_{n}=\left\{k\leq n: k\in K \right\}$. The natural density of $K$
is defined by $\delta(K)=\lim_{n}\frac{1}{n}|K_{n}|$ provided that
the limit exists, where $|K_{n}|$ denotes the cardinality of the set
$K_{n}$. A sequence $x=(x_{k})$ is called statistically convergent
to the number $\ell \in\mathbb{R}$, denoted by $st-\lim x=\ell$. For
each $\epsilon >0,$ the set
$K_{\varepsilon}=\left\{k\in\mathbb{N}:|x_{k}-\ell|\geq\epsilon
\right\} $ has a natural density zero, that is

\begin{eqnarray*}
\lim_{n \to \infty}\frac{1}{n}\left|\{k\leq n: |x_{k}-\ell|\geq
\epsilon \}\right|=0.
\end{eqnarray*}
 It is well know that every statistically convergence sequence is ordinary convergent, but the
converse is not true. The concept of statistical convergence was
firstly used in approximation theory by Gadjiev and Orhan \cite{10}.
They proved the Bohman$-$Korovkin type approximation theorem for
statistical convergence.

A real function $\rho_{0}$ is called a weighted function if it is
continuous on $\mathbb{R}$ and $\lim_{|x|\to 0}\rho_{0}(x)=\infty$,
$\rho_{0}(x)\geq 1$ for all $x\in\mathbb{R}$, where
$\rho_{0}(x)=1+x^{2}$. Let $B_{\rho_{0}}[0,\infty)$ be the set of
all functions $f$ defined on $[0,\infty)$ satisfying the condition
$|f(x)|\leq M_{f}(1+x^{2})$, where $M_{f}$ is a constant depending
on $f$. By $C_{\rho_{0}}[0,\infty)$, we denote the subspace of all
continuous function belonging to $B_{\rho_{0}}[0,\infty)$. Also,
$C^{\ast}_{\rho_{_{0}}}[0,\infty)$ be subspace of all continuous
functions $f\in C_{\rho_{0}}[0,\infty)$, for which $\lim_{|x| \to
\infty}\frac{f(x)}{1+x^{2}}$ is finite.  The norm on
$C^{\ast}_{\rho_{0}}[0,\infty)$ is defined as follows:
\begin{eqnarray*}
\parallel f\parallel_{\rho}=\sup_{x\in
[0,\infty)}\frac{|f(x)|}{1+x^{2}}.
\end{eqnarray*}

Now, we may begin the following lemma which is needed proving our
main result.
\begin{lem}\label{lemma3.1}
For $n\in \mathbb{N}$, $x\in [0,\infty)$ and $0<q<1$, we have%
\begin{eqnarray}
P^{\ast}_{n}\left( e_{0};q;x\right)& =&1,\label{3}\\
P^{\ast}_{n}\left( e_{1};q;x\right)
&=&x+\frac{D_{q}(A(1))E_{q}^{-\frac{[n]_{q}}{b_{n}}x}e_{q}^{\frac{q[n]_{q}}{b_{n}}x}}{A(1)}\frac{b_{n}}{[n]_{q}},\label{4}\\
P^{\ast}_{n}\left( e_{2};q;x\right)&
=&x^{2}+\frac{E_{q}^{-\frac{[n]_{q}}{b_{n}}x}e_{q}^{q\frac{[n]_{q}}{b_{n}}x}\left[qD_{q}(A(q))+D_{q}(A(1))\right]}{A(1)}\frac{xb_{n}}{[n]_{q}}+
\frac{E_{q}^{-\frac{[n]_{q}}{b_{n}}x}e_{q}^{q\frac{[n]_{q}}{b_{n}}x}D^{2}_{q}(A(1))}{A(1)}\frac{b^{2}_{n}}{[n]^{2}_{q}}
\label{5},
\end{eqnarray} where
 $e_i(x)=x^i$, $i=0, 1, 2$.
\end{lem}

\begin{proof}
By (\ref{1}) and  definition of $q$ derivative, we obtain that

\begin{eqnarray}
\sum_{k=0}^{\infty }\frac{p_{k}\left(
q;\frac{[n]_{q}}{b_{n}}x\right) }{\left[ k\right] _{q}!}&=&A\left(
1\right) e_{q}^{\frac{[n]_{q}}{b_{n}}x},  \label{6}\\
\sum_{k=0}^{\infty }\frac{p_{k}\left(
q;\frac{[n]_{q}}{b_{n}}x\right) }{\left[ k\right] _{q}!}\left[
k\right] _{q}&=&A\left( 1\right)
\frac{[n]_{q}}{b_{n}}xe_{q}^{\frac{[n]_{q}}{b_{n}}x}+e_{q}^{q\frac{[n]_{q}}{b_{n}}x}D_{q}A\left(
1\right) \label{7}\\
\sum_{k=0}^{\infty }\frac{p_{k}\left(
q;\frac{[n]_{q}}{b_{n}}x\right) }{\left[ k\right] _{q}!}\left[
k\right]_{q}^{2}&=&D_{q}^{2}(A(1))e_{q}^{\frac{q[n]_{q}}{b_{n}}x}+
\frac{x[n]_{q}}{b_{n}}e_{q}^{\frac{q[n]_{q}}{b_{n}}x}\left[qD_{q}(A(q))+D_{q}(A(1)\right]\label{8}\\&&+A(1)
\frac{x^{2}[n]_{q}^{2}}{b_{n}^{2}}
e_{q}^{\frac{[n]_{q}}{b_{n}}x}\nonumber .
\end{eqnarray}

By using the relations (\ref{6})-(\ref{8}), from (\ref{A}), we
obtain the results

\begin{equation*}
P^{\ast}_{n}\left( e_{0};q;x\right)
=\frac{E_{q}^{-\frac{[n]_{q}}{b_{n}}x}}{A\left( 1\right) }
\sum_{k=0}^{\infty }\frac{p_{k}\left(
q;\frac{[n]_{q}}{b_{n}}x\right) }{\left[ k\right]
_{q}!}=\frac{E_{q}^{-\frac{[n]_{q}}{b_{n}}x}}{A\left( 1\right)
}A\left( 1\right) e_{q}^{\frac{[n]_{q}}{b_{n}}x}=1,
\end{equation*}%
\begin{eqnarray*}
P^{\ast}_{n}\left( e_{1};q;x\right)
&=&\frac{E_{q}^{-\frac{[n]_{q}}{b_{n}}x}}{A\left( 1\right)
}\sum_{k=0}^{\infty } \frac{p_{k}\left(q;\frac{[n]_{q}}{b_{n}}x\right) }{\left[ k\right] _{q}!}\left(
\frac{\left[ k\right] _{q} }{\left[ n\right] _{q}}b_{n}\right) \\
&=&\frac{E_{q}^{-\frac{[n]_{q}}{b_{n}}x}b_{n}}{A\left( 1\right)
\left[ n\right] _{q} } \sum_{k=0}^{\infty }\frac{p_{k}\left(
q;\left[ n\right] _{q}t\right)}{\left[k\right] _{q}!}\left[ k\right] _{q} \\
&=&\frac{E_{q}^{-\frac{[n]_{q}}{b_{n}}x}b_{n}}{A\left( 1\right)
\left[ n\right] _{q} }\left( A\left( 1\right) \frac{[n]_{q}}{b_{n}}xe_{q}^{\frac{[n]_{q}}{b_{n}}x}+e_{q}^{q\frac{[n]_{q}}{b_{n}}x}D_{q}A\left( 1\right) \right) \\
&=&x+\frac{D_{q}(A(1))E_{q}^{-\frac{[n]_{q}}{b_{n}}x}e_{q}^{\frac{q[n]_{q}}{b_{n}}x}}{A(1)}\frac{b_{n}}{[n]_{q}}
\end{eqnarray*}

and
\begin{eqnarray*}
P^{\ast}_{n}\left( e_{2};q;x\right)
&=&\frac{E_{q}^{-\frac{[n]_{q}}{b_{n}}x}}{A\left( 1\right) }
\sum_{k=0}^{\infty}\frac{P_{k}\left(q;\frac{[n]_{q}}{b_{n}}x\right)}{\left[
k\right] _{q}!}\frac{[ k]^{2} _{q} }{[ n]^{2} _{q} }b_{n}^{2} \\
&=&\frac{E_{q}^{-\frac{[n]_{q}}{b_{n}}x}}{A\left(
1\right)}\frac{b^{2}_{n}}{[n]^{2}_{q}}\bigg\{D_{q}^{2}(
A(1))e_{q}^{\frac{q[n]_{q}}{b_{n}}x}+
e_{q}^{\frac{q[n]_{q}}{b_{n}}x}\frac{[n]_{q}}{b_{n}}x\left[qD_{q}(A(q))+D_{q}(A(1)\right]+A(1)
\frac{[n]_{q}^{2}}{b_{n}^{2}}
x^{2}e_{q}^{\frac{[n]_{q}}{b_{n}}x}\bigg\}\\
&=&x^{2}+\frac{E_{q}^{-\frac{[n]_{q}}{b_{n}}x}e_{q}^{q\frac{[n]_{q}}{b_{n}}x}\left[qD_{q}(A(q))+D_{q}(A(1))\right]}{A(1)}\frac{xb_{n}}{[n]_{q}}+
\frac{E_{q}^{-\frac{[n]_{q}}{b_{n}}x}e_{q}^{q\frac{[n]_{q}}{b_{n}}x}D^{2}_{q}(A(1))}{A(1)}\frac{b^{2}_{n}}{[n]^{2}_{q}}
\end{eqnarray*} Hence, the proof is completed.
\end{proof}

\begin{thm}
Assume that $q:=(q_n)$, $0<q_n<1$ be a sequence satisfying the following conditions:
\begin{eqnarray}\label{y1}
st-\lim_n \frac{b_{n}}{[n]_{q}}=0,\ \  st-\lim_n q_n=b.
\end{eqnarray}
Let $f\in C[0,\infty)\cap E$, where $E=\{f\in C[0,\infty):|f(x)|\leq
\alpha e^{\beta x}\,\, \textrm{for some}\,\,
\alpha\in\mathbb{R^{+}},\beta\in\mathbb{R} \}$. Then we have the
following:
\begin{equation*}\label{y2}
st-\lim_n \parallel P^{\ast}_{n}(f,q_n;.)-f
\parallel_{\rho_{0}}=0.
\end{equation*}

\end{thm}

\begin{proof}
It is enough to prove that
\begin{equation*}\label{y3}
st-\lim_n \parallel
P^{\ast}_{n}(e_v,q_n;.)-e_v\parallel_{\rho_{0}}=0
\end{equation*}
where $v=0,1,2.$

From the equation (\ref{3}), it is easy to obtain that
\begin{equation*}\label{y4}
st-\lim_n \parallel
P^{\ast}_{n}(e_0,q_n;.)-e_0\parallel_{\rho_{0}}=0.
\end{equation*}

 By (\ref{4}) and combining with $E_{q}^{-\frac{[n]_{q_{n}}}{b_{n}}}e_{q}^{\frac{q_{n}[n]_{q_{n}}}{b_{n}}}<1$,
we get
\begin{equation*}\label{y5}
\parallel P^{\ast}_{n}(e_1,q_n;.)-e_1\parallel_{\rho_{0}}\leq
\frac{D_{q}(A(1))}{A(1)}\frac{b_{n}}{[n]_{q_{n}}}
\parallel e_{0}\parallel_{\rho_{0}}.
\end{equation*}

For $\epsilon>0$, define the following sets:
\begin{eqnarray*}\label{y6}
\mathcal{M}&:=& \left\lbrace  k: \parallel
P^{\ast}_{n}(e_1,q_k;.)-e_1
\parallel_{\rho_{0}} \geq \epsilon
\right\rbrace,\\
\mathcal{M}_1&:=& \left\lbrace  k:
\frac{D_{q}(A(1))}{A(1)}\frac{b_{k}}{[k]_{q_{k}}}\geq \epsilon
\right\rbrace
\end{eqnarray*} such that $\mathcal{M}\subseteq \mathcal{M}_1 $.

By (\ref{y1}), one can write the following
\begin{equation*} \label{y9}
\delta \left\lbrace k \leq n : \parallel
P^{\ast}_{n}(e_1,q_k;.)-e_1\parallel_{\rho_{0}} \geq \epsilon
\right\rbrace \leq \delta \left\lbrace k \leq n:
\frac{D_{q}(A(1))}{A(1)}\frac{b_{k}}{[k]_{q_{k}}}\geq \epsilon
\right\rbrace
\end{equation*}

Hence, we get
\begin{equation} \label{y10}
st-\lim_n \parallel
P^{\ast}_{n}(e_1,q_n;.)-e_1\parallel_{\rho_{0}}=0.
\end{equation}

By (\ref{5}), one can see that
\begin{eqnarray*} \label{y11}
\parallel P^{\ast}_{n}(e_2,q_n;.)-e_2\parallel_{\rho_{0}}& \leq& \frac{
q_{_{n}}D_{q}(A(q))+D_{q}(A(1))}{A(1)}\frac{b_{n}}{[n]_{q_{n}}}\parallel
e_{1}\parallel_{\rho_{0}}\\&&+\frac{D_{q}(A(1))+D^{2}_{q}(A(1))}{A(1)}\frac{b^{2}_{n}}{[n]^{2}_{q_{n}}}\parallel
e_{0}\parallel_{\rho_{0}}
\end{eqnarray*}

Now, let $\epsilon>0$ be given, we define the following sets:
\begin{eqnarray*}\label{y12}
\mathcal{V}&:=& \left\lbrace  k: \parallel
P^{\ast}_{n}(e_2,q_k;.)-e_2
\parallel_{\rho_{0}} \geq \epsilon
\right\rbrace ,\\
\mathcal{V}_1&:=& \left\lbrace k:\frac{
q_{k}D_{q}(A(q))+D_{q}(A(1))}{A(1)}\frac{b_{k}}{[k]_{q_{k}}}\geq
\frac{\epsilon}{2}\right\rbrace ,\\
\mathcal{V}_2&:=& \left\lbrace  k:
\frac{D_{q}(A(1))+D^{2}_{q}(A(1))}{A(1)}
\frac{b^{2}_{k}}{[k]^{2}_{q_{k}}} \geq \frac{\epsilon}{2} \right\rbrace,\\
\end{eqnarray*}
such that $\mathcal{V}\subseteq \mathcal{V}_1 \cup \mathcal{V}_2 $.

Thus, we obtain
\begin{eqnarray}\label{y15}
\delta \left\lbrace k \leq n : \parallel
P^{\ast}_{n}(e_2,q_k;.)-e_2\parallel_{\rho_{0}} \geq \epsilon
\right\rbrace &\leq & \delta \left\lbrace k \leq n: \frac{
q_{k}D_{q}(A(q))+D_{q}(A(1))}{A(1)}\frac{b_{k}}{[k]_{q_{k}}}\right\rbrace
\nonumber\\
&+& \delta \left\lbrace k \leq
n:\frac{D_{q}(A(1))+D^{2}_{q}(A(1))}{A(1)}
\frac{b^{2}_{k}}{[k]^{2}_{q_{k}}} \geq
\frac{\epsilon}{2}\right\rbrace
\end{eqnarray}

Hence, (\ref{y1}) and (\ref{y15}) imply that
\begin{equation}\label{y16}
st-\lim_n \parallel
P^{\ast}_{n}(e_2,q_n;.)-e_2\parallel_{\rho_{0}}=0.
\end{equation}
\end{proof}

\section{Rates of statistical convergence}
In this section, we give the rates of statistical convergence of the
operators $P^{\ast}_{n}\left( f;q;x\right)$ by means of modulus of
continuity with the help of functions from Lipschitz class.

 $C_{B}[0,\infty)$ denotes of all real valued continuous bounded
functions. The modulus of continuity for the function $f\in
C_{B}[0,\infty)$ is defined as
\begin{equation*}
\omega(f,\delta)_{{\rho_{0}}}=\sup_{\substack{|x-y|\leq\delta \\
x,y\in [0,\infty)}}\frac{|f(x)-f(y)|}{1+x^{2+\lambda}}
\end{equation*}
where $\omega(f,\delta)_{{\rho_{0}}}$ for $\delta>0$, $\lambda\geq
0$ satisfy the following conditions: for every $f\in
C_{B}[0,\infty)$.
\begin{equation*}
\lim_{\delta\to \infty}\omega(f,\delta)_{{\rho_{0}}}=0
\end{equation*}
and
\begin{equation}\label{r1}
|f(x)-f(y)|\leq
\omega(f,\delta)_{{\rho_{0}}}\left(\frac{|x-y|}{\delta}+1\right).
\end{equation}
Now, we prove the following theorem for the rate of pointwise
convergence of the operators $P^{\ast}_{n}\left( f;q;x\right)$ to
the function $f(x)$ by means of modulus of continuity.
\begin{thm}\label{thm4.1}
 $q:=(q_{n})$ be sequence and $x\in
[0,\infty)$, then we have
\begin{equation*}\label{rr1}
|P^{\ast}_{n}\left( f;q;x\right)-f(x)|\leq
2\omega(f,\sqrt{\delta_{n}})_{\rho_{0}},
\end{equation*}
for all  $f\in C_{B}[0,\infty)\cap E$, where
\begin{equation}\label{ali}
\delta_{n}=\frac{
q_{n}D_{q}(A(q))+D_{q}(A(1))}{A(1)}\frac{b_{n}}{[n]_{q_{n}}}\parallel
e_{1}\parallel_{\rho_{0}}+\frac{D^{2}_{q}(A(1))}{A(1)}
\frac{b^{2}_{n}}{[n]^{2}_{q_{n}}}\parallel e_{0}\parallel_{\rho_{0}}
\end{equation}

\end{thm}
\begin{proof}
By the linearity and positivity of the operators $P^{\ast}_{n}\left(
f;q;x\right)$, one can write
\begin{eqnarray*}
|P^{\ast}_{n}\left( f;q;x\right)-f(x)| &\leq&\frac{E_{q}^{-\frac{[n]_{q}}{b_{n}}x}}{A\left( 1\right) }
\sum_{k=0}^{\infty }\frac{p_{k}\left( q;\frac{[n]_{q}}{b_{n}}x\right) }{\left[ k\right] _{q}!}\left |f\left(\frac{\left[ k\right] _{q}}
{\left[ n\right] _{q} }b_{n}\right)-f(x)\right | \\
&\leq& \frac{E_{q}^{-\frac{[n]_{q}}{b_{n}}x}}{A\left( 1\right)
}\sum_{k=0}^{\infty }\frac{p_{k}\left( q;\frac{[n]_{q}}{b_{n}}x\right) }{\left[ k\right] _{q}!}\omega(f,\delta)\left\{\frac{1}{\delta}\left | \frac{\left[ k%
\right] _{q} }{\left[ n\right] _{q}}b_{n}-x\right |+1\right\}\\
&=&\left\{\frac{1}{\delta}\frac{E_{q}^{-\frac{[n]_{q}}{b_{n}}x}}{A\left(
1\right) }\sum_{k=0}^{\infty }\frac{p_{k}\left(
q;\frac{[n]_{q}}{b_{n}}x\right) }{\left[ k\right]_{q}!}\left |
\frac{\left[ k\right] _{q}}{\left[ n\right] _{q} }b_{n}-x\right
|+1\right\}\omega(f,\delta) .
\end{eqnarray*}
If we apply Cauchy-Schwarz inequality for sums, we obtain
\begin{eqnarray*}
\sum_{k=0}^{\infty }\frac{p_{k}\left(
q;\frac{[n]_{q}}{b_{n}}x\right) }{\left[ k\right] _{q}!}\left |
\frac{\left[ k\right] _{q} }{\left[ n\right] _{q}}b_{n}-x\right|^{2}\leq\left(\sum_{k=0}^{\infty }\frac{p_{k}\left( q;%
\left[ n\right] _{q}t\right) }{\left[ k\right] _{q}!} \right)^{1/2}
\left(\sum_{k=0}^{\infty }\frac{p_{k}\left(
q;\frac{[n]_{q}}{b_{n}}x\right) }{\left[ k\right] _{q}!}\left (
\frac{\left[ k \right] _{q} }{\left[ n\right] _{q}
}b_{n}-x\right)^{2}\right)^{1/2}.
\end{eqnarray*}
Using above inequality and by (\ref{3}), we have
\begin{eqnarray*}
|P^{\ast}_{n}\left( f;q;x\right)-f(x)|
&\leq&\omega(f,\delta)\left\{1+\frac{1}{\delta}\left[\frac{E_{q}^{-\frac{[n]_{q}}{b_{n}}x}}{A\left(
1\right) } \sum_{k=0}^{\infty }\frac{p_{k}\left(
q;\frac{[n]_{q}}{b_{n}}x\right) }{\left[ k\right] _{q}!}\left (
\frac{\left[ k\right] _{q}}{\left[ n\right] _{q}}b_{n}-x\right)^{2}\right]^{1/2}\right\}\\
&=&\omega(f,\delta)\left\{1+\frac{1}{\delta}
\left[\frac{E_{q}^{-\frac{[n]_{q}}{b_{n}}x}}{A\left( 1\right)
}\sum_{k=0}^{\infty }\frac{p_{k}\left(
q;\frac{[n]_{q}}{b_{n}}x\right) }{\left[ k\right] _{q}!}\left (
\frac{\left[ k\right] _{q} }{\left[ n\right] _{q}}b_{n}-x\right)^{2}\right]^{1/2}\right\}\\
 &=&\omega(f,\delta)\left\{1+\frac{1}{\delta}\left[P^{\ast}_{n}\left(
 s-e_{1})^{2};q;x\right)\right]^{1/2}\right\}.
\end{eqnarray*}
From Lemma \ref{lemma3.1} and the fact
$E_{q}^{-\frac{[n]_{q_{n}}}{b_{n}}}e_{q}^{\frac{q_{n}[n]_{q_{n}}}{b_{n}}}<1$,
we have
\begin{eqnarray}\label{a33}
P^{\ast}_{n}\left( (s-e_{1})^{2};q_{n};x\right)\leq\frac{
q_{n}D_{q}(A(q))+D_{q}(A(1))}{A(1)}\frac{b_{n}}{[n]_{q_{n}}}\parallel
e_{1}\parallel_{\rho_{0}}+\frac{D^{2}_{q}(A(1))}{A(1)}
\frac{b^{2}_{n}}{[n]^{2}_{q_{n}}}\parallel
e_{0}\parallel_{\rho_{0}}.
\end{eqnarray}
By \eqref{a33} and we choose $\delta=\sqrt{\delta_{n}}$, we get
\begin{equation*}\label{rr1}
|P^{\ast}_{n}\left( f;q;x\right)-f(x)|\leq
2\omega(f,\sqrt{\delta_{n}})_{\rho_{0}}.
\end{equation*}
 This step concludes the proof.
\end{proof}

 We know that a function $f\in C[0,\infty)$  is in
 $Lip_{M}(\alpha)$ on $F$, $0<\alpha\leq 1$, $F$ is a any bounded subset of the interval $[0,\infty)$ if it is satisfies the conditions

\begin{equation}\label{a3}
\left|f(y)-f(x) \right|\leq
M|x-y|^{\alpha},\,\,y\in[0,\infty)\,\,\textrm{and}\,\,x\in F,
\end{equation}
where $M$ is a constant depending only $\alpha$ and $f$.

\begin{thm}
Let $f \in C[0,\infty)\cap Lip_{M}(\alpha)$, $0<\alpha\leq 1$. Then,
we get
\begin{equation*}
|P^{\ast}_{n}(f;q_n;x)-f(x)|\leq
M\left(\delta^{\alpha/2}_{n}+d^{\alpha}(x,F)\right),\,\,x\in[0,\infty),
\end{equation*}
where $M$ is a constant depending only $\alpha$ and $f$ and $d(x,F)$
is the distance between $x$ and $F$ defined by $d(x,F)=\inf\{|x-t|:
t\in F\}$.
\end{thm}
\begin{proof}
Let $\overline{F}$ be the closure of $F$ in $[0,\infty)$. The there
exists al leat point $x_{0}\in \overline{F}$ such that
$d(x,F)=|x-x_{0}|$. By our assumption and monotonicity of
$P^{\ast}_{n}\left( f;q_{n};x\right)$, one can write the following

\begin{eqnarray*}
|P^{\ast}_{n}(f;q_n;x)-f(x)|&\leq &
P^{\ast}_{n}\left(|f(t)-f(x)|;q_n;x\right)+P^{\ast}_{n}\left(|f(x)-f(x_{0})|;q_n;x\right)\\
&\leq&
M\{P^{\ast}_{n}\left(|t-x|^{\alpha};q_n;x\right)+|x-x_{0}|^{\alpha}\}.
\end{eqnarray*}
Next, we applying the H\"{o}lder inequality with
$p=\frac{1}{\alpha}, q=\frac{\alpha}{2-\alpha}$ and we get

\begin{eqnarray*}
|P^{\ast}_{n}(f;q_n;x)-f(x)|&\leq &
M\{P^{\ast}_{n}\left[|t-x|^{2};q_n;x\right]^{\alpha/2}+|x-x_{0}|^{\alpha}\}.
\end{eqnarray*}
This step conclude the proof.
\end{proof}
Now, we obtain local direct estimate of the operators
$P^{\ast}_{n}\left( f;q;x\right)$ using the Lipschitz-type maximal
function of order $\alpha$ introduce Lenz \cite{lenz} as
\begin{eqnarray}\label{a4}
\omega_{\alpha}(f,x)=\sup_{t\neq x,\,\,t\in
[0,\infty)}\frac{|f(t)-f(x)|}{|t-x|^{\alpha}},\,\,x\in[0,\infty)\,\,\textrm{and}\,\,\alpha\in(0,1].
\end{eqnarray}

\begin{thm}
Let $f \in Lip_{M}(\alpha)$, $0<\alpha\leq 1$, then we have
\begin{equation*}
|P^{\ast}_{n}(f;q_n;.)-f(.)|\leq
\omega_{\alpha}(f,x)\delta^{\alpha/2}_{n},
\end{equation*}
where  $\delta_{n}$ defined in \eqref{ali}.
\end{thm}

\begin{proof}
From (\ref{a4}), we obtain
\begin{eqnarray*}
|P^{\ast}_{n}(f;q_n;x)-f(x)|&\leq&
P^{\ast}_{n}(|f(t)-f(x)|,q;x)\\
&\leq&\omega_{\alpha}(f,x) P^{\ast}_{n}(|t-x|^{\alpha},q;x).
\end{eqnarray*}
Again for $p=\frac{1}{\alpha}, q=\frac{\alpha}{2-\alpha}$, applying
the H\"{o}lder  inequality, we get
\begin{eqnarray*}
|P^{\ast}_{n}(f;q_n;x)-f(x)|&\leq&
P^{\ast}_{n}(|f(t)-f(x)|,q;x)\\
&\leq& \omega_{\alpha}(f,x)
\{P^{\ast}_{n}(e_{1}-x)^{2},q;x)\}^{\alpha/2}=\omega_{\alpha}(f,x)\delta^{\alpha/2}_{n}.
\end{eqnarray*}
Hence, we get the desired result.
\end{proof}

\section{Local Approximation}
In this section, we state the local approximation theorem of the
operators $P^{\ast}_{n}\left(f;q;x \right)$. Let $C_B\left[0,\infty
\right)$ be the space of all real valued continuous bounded
functions $f$ on $\left[0,\infty \right)$ with the norm $\parallel f
\parallel=\sup\left\lbrace \vert f(x)\vert :
x\in[0,\infty)\right\rbrace $. The K-functional of $f$ is defined by
\begin{eqnarray*}
K_2(f;\delta)=\inf_{g\in W^2}\left\lbrace \Vert f-g\Vert +
\delta\Vert g'' \Vert\right\rbrace,
\end{eqnarray*}
where $\delta>0$ and $W^2=\left\lbrace g\in C_B\left[0,\infty \right): g', g''\in C_B[0,\infty) \right\rbrace $. By
Devore-Lorentz \cite[p. 177]{lorenz}, there exists an absolute
constant $C>0$ such that
\begin{eqnarray} \label{5.1}
K_2\left(f,\delta\right)\leqslant C \omega_2\left(f,\sqrt{\delta}\right)
\end{eqnarray}
where
\begin{eqnarray*}
\omega_2\left(f,\sqrt{\delta}\right)=\sup_{0<h\leq 0} \sup_{x\in
[0,\infty)}\left\vert f(x+2h)-2f(x+h)+f(x)\right\vert
\end{eqnarray*}
is the second order modulus of smoothness of $f$. Moreover,
\begin{eqnarray*}
\omega(f,\delta)=\sup_{0<h\leq 0} \sup_{x\in [0,\infty)}\vert
f(x+h)-f(x) \vert
\end{eqnarray*}
denotes the modulus of continuity of $f$.

Now, we give the direct local approximation theorem for the
operators $P^{\ast}_{n}\left(f;q;x \right)$.

\begin{thm}
Let $q \in (0,1)$. We have
\begin{eqnarray*}
\vert P^{\ast}_{n}\left(f;q;x \right)-f(x)\vert \leq K
\omega_2\left(f,\varphi_{n} \right)+
\omega\left(f,\frac{D_{q}(A(1))}{A(1)}\frac{b_{n}}{[n]_{q}}\right)
\end{eqnarray*}
$\forall x\in $ $[0,\infty)$, $f\in C_B\left[0,\infty \right)$,
where $K$ is a positive constant and
\begin{equation*}\label{ali}
\varphi_{n}=\frac{
q_{n}D_{q}(A(q))+D_{q}(A(1))}{A(1)}\frac{b_{n}}{[n]_{q_{n}}}+\frac{A(1)D^{2}_{q}(A(1))+[D_{q}(A(1))]^{2}}{A^{2}(1)}
\frac{b^{2}_{n}}{[n]^{2}_{q_{n}}}
\end{equation*}

\end{thm}

\begin{proof}
Let us define the following operators
\begin{eqnarray}\label{5.2}
\widetilde{P}^{\ast}_{n}\left(f;q;x \right)= P^{\ast}_{n}\left(f;q;x
\right)-f\left(x+\frac{D_{q}(A(1))E_{q}^{-\frac{[n]_{q}}{b_{n}}x}e_{q}^{\frac{q[n]_{q}}{b_{n}}x}}{A(1)}\frac{b_{n}}{[n]_{q}}
\right)+f(x),
\end{eqnarray}
 $x\in [0,\infty)$. The operators $\widetilde{P}^{\ast}_{n}\left(f;q;x \right)$ are linear. Thus, we have the following:
 \begin{eqnarray}\label{5.3}
\widetilde{P}^{\ast}_{n}\left(s-x;q;x \right)=0,
 \end{eqnarray}
 (see Lemma  \ref{lemma3.1}).
Let $g\in W^2$, from Taylor's expansion
\begin{eqnarray*}
g(s)=g(x)+g'(x)(s-x)+ \int_{x}^s (s-u)g''(x)du,
\end{eqnarray*}
$s\in[0,\infty)$ and (\ref{5.3}) we obtain
\begin{eqnarray*}
\widetilde{P}^{\ast}_{n}\left(g;q;x \right)=g(x)+
\widetilde{P}^{\ast}_{n}\left(\int_{x}^s (s-u)g''(x)du \right).
\end{eqnarray*}

By (\ref{5.2}), we have the following
\begin{eqnarray}\label{5.4}
\vert \widetilde{P}^{\ast}_{n}\left(g;q;x\right)-g(x) \vert
&\leqslant & \left\vert P^{\ast}_{n}\left(\int_{x}^s (s-u)g''(u)du
\right)\right\vert\nonumber\\ &&+ \left\vert
\int_{x}^{x+\frac{D_{q}(A(1))E_{q}^{-\frac{[n]_{q}}{b_{n}}x}e_{q}^{\frac{q[n]_{q}}{b_{n}}x}}{A(1)}\frac{b_{n}}{[n]_{q}}}
\left(x+\frac{D_{q}(A(1))E_{q}^{-\frac{[n]_{q}}{b_{n}}x}e_{q}^{\frac{q[n]_{q}}{b_{n}}x}}{A(1)}\frac{b_{n}}{[n]_{q}}-u\right)g''(u)du \right\vert \nonumber \\
&\leqslant & P^{\ast}_{n} \left(\left\vert \int_{x}^s
(s-u)g''(u)du\right\vert ,x \right)\nonumber \\
&&+\int_{x}^{x+\frac{D_{q}(A(1))E_{q}^{-\frac{[n]_{q}}{b_{n}}x}e_{q}^{\frac{q[n]_{q}}{b_{n}}x}}{A(1)}\frac{b_{n}}{[n]_{q}}}
\left\vert
x+\frac{D_{q}(A(1))E_{q}^{-\frac{[n]_{q}}{b_{n}}x}e_{q}^{\frac{q[n]_{q}}{b_{n}}x}}{A(1)}\frac{b_{n}}{[n]_{q}}-u
\right\vert \left\vert g''(u)\right\vert du \nonumber \\
&\leq & \left[ P^{\ast}_{n}\left(s-x\right)^2+
\left(\frac{D_{q}(A(1))E_{q}^{-\frac{[n]_{q}}{b_{n}}x}e_{q}^{\frac{q[n]_{q}}{b_{n}}x}}{A(1)}\frac{b_{n}}{[n]_{q}}
\right)^2 \right]\Vert g'' \Vert
\end{eqnarray}

Therefore, from (\ref{5.4}), we obtain
\begin{eqnarray}\label{5.5}
\vert \widetilde{P}^{\ast}_{n}\left(g;q;x\right)-g(x) \vert \leq
\varphi_{n}.
\end{eqnarray}
By (\ref{A}), (\ref{lemma3.1}) and (\ref{5.2}), we get
\begin{eqnarray}\label{5.6}
\vert \widetilde{P}^{\ast}_{n}\left(f;q;x\right) \vert &\leqslant & P^{\ast}_{n}\left(f;q;x\right)+2\Vert f \Vert \nonumber \\
&\leqslant & \Vert f \Vert P^{\ast}_{n}\left(1;q;x\right)+2\Vert f \Vert \nonumber \\
&\leqslant & 3\Vert f \Vert
\end{eqnarray}

and by (\ref{5.2}), (\ref{5.5}) and (\ref{5.6})
\begin{eqnarray*}
\vert P^{\ast}_{n}\left(f;q;x\right)-f(x) \vert &\leqslant & \vert
\widetilde{P}^{\ast}_{n}\left(f-g;q;x\right)-(f-g)(x) \vert
+\vert\widetilde{P}^{\ast}_{n}\left(g;q;x\right)-g(x) \vert \\
&&+ \left\vert f\left(x+\frac{D_{q}(A(1))E_{q}^{-\frac{[n]_{q}}{b_{n}}x}e_{q}^{\frac{q[n]_{q}}{b_{n}}x}}{A(1)}\frac{b_{n}}{[n]_{q}} \right)-f(x) \right\vert \\
&\leqslant & 4\Vert f-g \Vert + \varphi_{n}\Vert g'' \Vert.
\end{eqnarray*}
From (\ref{5.1}), one can see that
\begin{eqnarray*}
\vert P^{\ast}_{n}\left(f;q;x\right)-f(x) \vert\leq K
\omega_2\left(f,\varphi_{n} \right)+
\omega\left(f,\frac{D_{q}(A(1))}{A(1)}\frac{b_{n}}{[n]_{q}}\right).
\end{eqnarray*}
and this concludes the proof.
\end{proof}

\end{document}